\documentclass[12pt]{amsart}
\usepackage{amsmath, amsfonts, amsthm, amssymb}  
\usepackage[margin=30mm]{geometry}

\usepackage{graphicx}

\usepackage{etoolbox}
\usepackage{enumerate}
\usepackage{listings}
\usepackage{booktabs}
\usepackage{float}

\newcommand{\bone}{\mathbb{I}}

\newcommand{\E}{\operatorname{\mathbb{E}}}
\renewcommand{\P}{\operatorname{\mathbb{P}}}

\newcommand{\Z}{\mathbb{Z}}
\newcommand{\calF}{\mathcal{F}}

\newcommand{\calH}{\mathcal{H}}

\newcommand{\ol}{\overline}

\newcommand{\eps}{\varepsilon}

\def\indented#1{\list{}{}\item[]}
\let\indented=\endlist

\newtheorem{theorem}{Theorem}[section]
\newtheorem{lemma}[theorem]{Lemma}

\theoremstyle{definition}

\newtheorem{problem}[theorem]{Problem}

\newtheorem{conjecture}[theorem]{Conjecture}
\newtheorem{remark}[theorem]{Remark}

\numberwithin{equation}{section}

\begin{document}

\title{On the size of earthworm's trail}

\author{Krzysztof Burdzy, Shi Feng and Daisuke Shiraishi}

\email{burdzy@uw.edu, fengshi@uw.edu}
\address{Department of Mathematics, University of Washington, Seattle,
WA 98195}
\email{shiraishi@acs.i.kyoto-u.ac.jp}
\address{Department of Advanced Mathematical Sciences, Graduate School of Informatics, Kyoto University, Kyoto, 606-8501}

\thanks{KB’s research was supported in part by Simons Foundation Grants 506732 and 928958. DS is supported by a JSPS Grant-in-Aid for Early-Career Scientists, 18K13425 and JSPS KAKENHI Grant Number 17H02849, 18H01123, 21H00989, 22H01128 and 22K03336.}

\keywords{Earthworm, random walk}

 \subjclass[2020]{60K40, 60J99} 


\maketitle

\begin{abstract}
    We investigate the number of holes created by an ``earthworm'' moving on the two-dimensional integer lattice. 
    The earthworm is modeled by a simple random walk. 
    At the initial time, all vertices are filled with grains of soil except for the position of the earthworm.
    At each step, the earthworm pushes the soil in the direction of its motion. 
    It leaves a hole (an empty vertex with no grain of soil) behind it.
    If there are holes in front of the earthworm (in the direction of its step),
    the closest hole is filled with a grain of soil. Thus the number of holes increases by 1 or remains unchanged
    at every step. We  show that the number of holes is at least $\mathcal{O}(n^{3/4})$ after $n$ steps.
    
\end{abstract}

\section{Introduction}

We will investigate the number of holes created by an ``earthworm'' moving on the two-dimensional integer lattice. 
    The earthworm is modeled by a simple random walk. 
    At the initial time, all vertices of $\Z^2$ are filled with grains of soil except for the position of the earthworm.
    At each step, the earthworm pushes the soil in the direction of its motion. 
    It leaves a hole (an empty vertex with no grain of soil) behind it.
    If there are holes in front of the earthworm (in the direction of its step),
    the closest hole is filled with a grain of soil.
See Section \ref{model} for the rigorous definition.

An earthworm model very similar to ours was  investigated in \cite{burdzy2013fractal}
except that the state space was a two-dimensional discrete torus. In that paper,
the number of holes was constant but their distribution changed over time and 
converged to the stationary regime. The ``dimension'' of the set of holes
was investigated---the definition was based on the local scaling properties
of the set of holes. 
A similar definition of the ``dimension'' of the set of holes was adopted
in \cite{wxml} but the state space was the whole $\Z^2$. In both cases, no theorems were proved---only the results
of simulations were reported. The dimension of the set of holes was close
to $3/2$ but appeared to be strictly larger than $3/2$.

In this article, we will prove a theorem supporting the conjecture that
the dimension of the set of holes is at least $3/2$ but our interpretation
of the ``dimension'' will be different. Namely, 
we will  show that the number of holes is at least $\mathcal{O}(n^{3/4})$ after $n$ steps
with high probability. See Theorem \ref{f24.1} for a precise statement.
Note that the exponent $3/4$ in our theorem is consistent with the ``dimension'' $3/2$,
in the sense that after $n^2$ steps, the random walk will visit about
$n^2$ vertices, on the same order as the number of sites in a square
with side $n$, and then the number of holes will be at least $\mathcal{O}(n^{3/2})$.

The model appears
to be very hard to analyze rigorously despite its simple nature.
We hope that our result, simulations, and open problems will inspire other
researchers.

The model and our main result will be presented in Section \ref{model}. 
This will be followed by Section \ref{proofs} containing the proofs.
Finally, Section \ref{open} will present some open problems and simulations.

After submitting the paper for publication, we learned from the anonymous Referee that our main result is very close to \cite[Lem. 2]{BenWil}. This is because ``tan points'' introduced in \cite{BenWil} are essentially the same as points $B_n^{\text{right}}$ defined in Section \ref{proofs} below.
In view of this preexisting research, the contributions of the present paper
include the rigorous introduction of the earthworm model and Conjecture \ref{con34} which, in our opinion, is challenging but not impossible to resolve. The proof of our main result,  Theorem \ref{f24.1} seems to be considerably simpler than that of \cite[Lem. 2]{BenWil}. The latter was based on hard estimates from \cite{BMS}. Our Theorem \ref{f24.1} is slightly stronger than \cite[Lem. 2]{BenWil} because $\eps>0$ in the statement of our theorem can be arbitrarily small.

\section{Model, notation and the main result}\label{model}

The earthworm is represented by a  simple random walk $X_n$ on $\mathbb{Z}^2$, starting at  $X_0 = (0,0)$. At every time $n$, every lattice point is in one of two states---either it is a hole or it is filled with a grain of soil. At time $n=0$,  $(0,0)$ is a hole and every other site is filled. Let $\mathcal{H}_n$ denote the set of holes at time $n$. Then $\mathcal{H}_0 = \{(0,0)\}$.

The process $(X_n,\calH_n) $ is Markov with the following dynamics.
Suppose $X_n = (x_n,y_n)$ and the earthworm goes to the right at the next step, i.e., $X_{n+1} = (x_n+1,y_n)$. We need to check if there is a hole to the right of $X_n$. Let 
\begin{align*}
    \mathcal{X}_n^{\text{right}} = \{(x,y_n) : x>x_n \}.
\end{align*}
If $\mathcal{X}_n^{\text{right}} \cap \mathcal{H}_n = \emptyset$ then we say that the earthworm created a hole at position $X_{n+1}=(x_n+1,y_n)$, and we let $\mathcal{H}_{n+1} = \mathcal{H}_{n} \cup \{X_{n+1}\}$. If $\mathcal{X}_n^{\text{right}} \cap \mathcal{H}_n \neq \emptyset$ and $X_{n+1}$ is not a hole then we let the earthworm create a hole at $X_{n+1}$,  push the soil in front of it,  and eliminate the nearest hole to the right of $X_{n+1}$. 
If $\mathcal{X}_n^{\text{right}} \cap \mathcal{H}_n \neq \emptyset$ and $X_{n+1}$ is a hole then no holes are created or annihilated.
More precisely, let $(x_*,y_n) \in \mathcal{X}_n^{\text{right}} \cap \mathcal{H}_n$ be the element with the smallest $x$-coordinate. Then let
\begin{align*}
    \mathcal{H}_{n+1} = (\mathcal{H}_{n} \backslash \{(x_*,y_n)\}) \cup \{X_{n+1}\}.
\end{align*}
In this case, we also say that the earthworm transferred the hole from $(x_*,y_n)$ to $X_{n+1}$. 

If the earthworm goes in any other direction at time $n+1$, the mechanism is the same with respect to that direction---check if there are any holes in front of the earthworm and update holes accordingly.
This completes the definition of the process $(X_n,\calH_n) $.

 Let  $\mathcal{F}_i$ be the $\sigma$-field generated by $X_0,X_1,...,X_i$. 
 Note that $\calH_i$ is $\calF_i$-measurable for every $i$.

The indicator random variable of an event $A$ will be denoted $\bone(A)$.
 
 Let $H_i$ be the indicator of the event that  the earthworm created an extra hole
 (i.e., that the number of holes increased by 1) at time $i$
 for $i\geq 1$. Set $H_0 = 1$. Let $S_n$  be the total number of holes at time $n$. Then,
\begin{align*}
    S_n = |\mathcal{H}_n| = \sum_{i=0}^{n} H_i.
\end{align*}

The goal of this paper is to provide a lower bound for $S_n$.  The following is our main result.
\begin{theorem}\label{f24.1}
For every $\eps>0$ there exists $\delta>0$ such that
\begin{align*}
    \liminf_{n\to\infty}\P(S_n \geq \delta n^{3/4}) \geq 1-\eps.
\end{align*}
\end{theorem}

\section{Proofs}\label{proofs}

We will first derive a lower bound for the expectation of $S_n$ using  a Beurling-type estimate for random walks. Then we will give an upper bound for the variance of $S_n$. 
We will combine these estimates using the ``second-moment method'' and the associated Paley-Zygmund inequality.
Finally, we will show that $S_n$ cannot fluctuate too much to fall below the order of its expectation.

\begin{lemma}\label{f25.1}
There exists $c_1>0$ such that
\begin{align*}
    \liminf_{n\to\infty}\E\left[n^{-3/4} S_n\right] \geq c_1.
\end{align*}
\end{lemma}
\begin{proof}
We start with two observations that are crucial for the proof. First, the earthworm creates a new hole at time $n+1$, i.e., $H_{n+1} = 1$, if and only if there is no hole in  the direction of its motion at time $n$. Second, for any  $(x,y) \in \mathbb{Z}^2$, if the earthworm did not visit $(x,y)$  before or at time $n$ then $(x,y)$ is not a hole at time $n$. Therefore
\begin{align}\label{f25.2}
    \{& H_{n+1}=1\} 
    = \{\text{There is no hole in the direction of the step at time } n\}\\ 
    &\supseteq \{\text{The earthworm did not visit any points}\notag\\
&\qquad \text{ in the direction of the step before or at time } n\}.
    \notag
\end{align}
We will define the last event in a more precise way and estimate its probability. 

Denote  $A_{n+1}^{\text{right}} = \{X_{n+1}-X_{n} = (1,0)\}$, i.e., $A_{n+1}^{\text{right}} $ is the event that $X$ goes to the right  at $(n+1)$st step. We will use a similar notation: $A_{n+1}^{\text{left}},A_{n+1}^{\text{up}},A_{n+1}^{\text{down}}$. For $i = \{\text{right, left, up, down}\}$,
\begin{align}\label{f25.3}
    \E[\bone(A_{n+1}^i)\mid \mathcal{F}_n] = 1/4.
\end{align}
Let $B_n^{\text{right}} = \{X_k-X_n \neq (a,0) \text{ for all } 0\leq k\leq n \text{ and } a \geq 1\}$, i.e., 
$B_n^{\text{right}}$ is the event that the random walk did not visit any point to the right of $X_n$ before time $n$. 
We will use the analogous notation: 
$B_{n}^{\text{left}},B_{n}^{\text{up}},B_{n}^{\text{down}}$. Note that $B_{n}^i \in \mathcal{F}_n$ for $i = \{\text{right, left, up, down}\}$.
The event on the right hand side of \eqref{f25.2} can be expressed as the union of four events as follows,
\begin{align*}
    \bigcup_{i=\{\text{left,right,up,down}\}} A_{n+1}^i \cap B_n^i.
\end{align*}
Therefore, conditioning on $\mathcal{F}_n$, applying \eqref{f25.2} and \eqref{f25.3} and using the symmetry of simple random walk, we obtain
\begin{align}\label{f25.4}
    \E[H_{n+1}]
    &\geq \sum_{i=\{\text{left,right,up,down}\}} \E[\bone(A_{n+1}^i) \bone(B_n^i)]\\
    &= \sum_{i=\{\text{left,right,up,down}\}}  \E[\bone(B_n^i)]/4
    = \E[\bone(B_n^{\text{right}})].\notag
\end{align}

For a fixed $n$, let $Y_i = X_{n-i}-X_n$ for $i=0,1,\dots, n$, and note that $Y$ is a simple random walk starting from $Y_0=(0,0)$. 
We denote $C_n^{\text{right}} = \{Y_k \neq (a,0) \text{ for all } 0\leq k\leq n \text{ and } a\geq 1\}$, i.e., $C_n^{\text{right}}$ is the event that $Y_0,\dots,Y_n$ did not visit the positive $x$-axis. Note that
$B_n^{\text{right}} = C_n^{\text{right}}$.
By \cite[(2.35)]{Lawler} there exists a constant $c_*>0$ such that
\begin{align}\label{m4.1}
    \liminf_{n\rightarrow \infty} n^{1/4}\E[\bone(C_n^{\text{right}})] > c_*.
\end{align}
This, the equality of events $B_n^{\text{right}} $ and $ C_n^{\text{right}}$, and \eqref{f25.4} imply that
\begin{align*}
\liminf_{n\rightarrow \infty} n^{1/4}\E[H_{n+1}] 
\geq \liminf_{n\rightarrow \infty} n^{1/4}\E[\bone(B_n^{\text{right}})]
  =  \liminf_{n\rightarrow \infty} n^{1/4}\E[\bone(C_n^{\text{right}})] > c_*.
\end{align*}
Therefore, using the integral approximation, for some $c_1 >0$,
\begin{align*}
    \liminf_{n\rightarrow \infty}\E\left[n^{-3/4}S_n\right] 
    &= \liminf_{n\rightarrow \infty} n^{-3/4} \sum_{i=1}^{n} \E[H_{i}] 
    \geq \lim_{n\rightarrow \infty} n^{-3/4} \int_{1}^{n} c_* x^{-1/4}\;dx
    = c_1 .
\end{align*}
\end{proof}

To bound the variance of $S_n$, we need to bound $\E[H_iH_j]$. Inequality \eqref{f25.6} in the following lemma will be used to derive an upper bound. Inequality \eqref{f25.7} will be used later.
\begin{lemma}
For any $0\leq i \leq j$,
\begin{align}\label{f25.6}
    \E[H_j \mid  \mathcal{F}_i] &\leq \E[H_{j-i}],\\
    \P(S_j\leq x \mid  \mathcal{F}_i) &\leq \P(S_{j-i} \leq x).\label{f25.7}
\end{align}
\end{lemma}

\begin{proof}
Fix $j\geq 0$ and $0\leq i \leq j$. 
We define $\{(X'_k,\calH'_k), k\geq i\}$ by setting $X_k' = X_k$ for all $k\geq i$ and $\calH'_i = \{X_i\} \subseteq \mathcal{H}_i$. 
The dynamics of $\calH'_k$, i.e., the mechanism of creating new holes,  is the same as for the original model. In other words, all holes created by $X$ before time $i$ have been erased but otherwise, $X'$ follows the same trajectory as that of $X$.

Let $H'_k $ be the indicator of the event that $X'$ creates a new hole at time $k$ for $k>i$. Let $H'_i = 1$.
 Note that
 $H'_k \in \mathcal{F}_k$ for $k\geq i$.
 As a consequence of the erasure of the initial holes created by $X$, the time shift, and the Markov property, we obtain
\begin{align}\label{f25.9}
    \E[H'_j \mid  \mathcal{F}_i] = \E[H_{j-i}].
\end{align}

It requires a moment's thought but it is totally elementary to see that  if $\mathcal{H}'_k \subseteq \mathcal{H}_k$ then $\mathcal{H}'_{k+1} \subseteq \mathcal{H}_{k+1}$, for $k\geq i$, since the creation of new holes is governed in both cases by identical steps of the two earthworms. 
Recall that $\calH'_i  \subseteq \mathcal{H}_i$. By induction, $\mathcal{H}'_k \subseteq \mathcal{H}_k$ for all $k\geq i$. Just before the $j$th step, $X$ and $X'$ are at the same location, and they will go in the same direction at step $j$. Since $\mathcal{H}'_{j-1} \subseteq \mathcal{H}_{j-1}$, if $X$ does not have a hole in the direction of the next step, $X'$ also does not have a hole in the direction of the next step. Therefore, if $H_j=1$ then $H'_j=1$. Hence, $ H_j\leq H_j'$.
Taking conditional expectation with respect to $\mathcal{F}_i$ and using \eqref{f25.9}, we obtain 
\begin{align*}
    \E[H_j\mid \mathcal{F}_i]\leq \E[H_j'\mid \mathcal{F}_i] = \E[H_{j-i}].
\end{align*}
This proves \eqref{f25.6}.\\

To prove \eqref{f25.7}, we denote $S'_n = \sum_{k=i}^{n} H'_k = |\mathcal{H}'_n|$. The following formula holds for the same reasons as those for \eqref{f25.9},
\begin{align*}
    \P(S'_j\leq x\mid \mathcal{F}_i) = \P(S_{j-i}\leq x).
\end{align*}
Since $\mathcal{H}'_j \subseteq \mathcal{H}_j$, we have $S_j' = |\mathcal{H}'_j| \leq |\mathcal{H}_j| = S_j$. Therefore, taking conditional expectation with respect to $\mathcal{F}_i$, we get
\begin{align*}
    \P(S_j\leq x \mid  \mathcal{F}_i) \leq \P(S'_j\leq x \mid  \mathcal{F}_i) = \P(S_{j-i}\leq x).
\end{align*}
\end{proof}

Given \eqref{f25.6}, we can bound the variance of $S_n$ by expanding it into the sum of $H_i$'s. Then we will use the Paley–Zygmund inequality to derive a lower bound for $S_n$.

\begin{lemma}\label{f25.10}
For all $n\geq 0$ and $0<\theta<1$,
\begin{align*}
    \P(S_n > \theta \E[S_n]) \geq \frac{(1-\theta)^2}{2}.
\end{align*}
\end{lemma}
\begin{proof}
By \eqref{f25.6}, for any $0\leq i\leq j$,
\begin{align*}
    \E[H_iH_j \mid  \mathcal{F}_i] 
    &= H_i\E[H_j\mid  \mathcal{F}_i] \leq H_i\E[H_{j-i}].
\end{align*}
Taking expectation on both sides, we get
\begin{align*}
    \E[H_iH_j] \leq \E[H_i]\E[H_{j-i}] .
\end{align*}
Therefore,
\begin{align*}
    \E[S_n^2] 
    &= \E\left[\left(\sum_{i=0}^{n} H_i\right)^2\right]
    \leq 2\sum_{i=0}^{n}\sum_{j=i}^{n} \E[H_iH_j]
    \leq 2\sum_{i=0}^{n}\sum_{j=i}^{i+n} \E[H_iH_j]\\
    &\leq 2\sum_{i=0}^{n}\sum_{j=0}^{n} \E[H_i]\E[H_j]
    = 2(\E[S_n])^2.
\end{align*}
By the Paley–Zygmund inequality
(see \cite[Sect. 5.1, Cor. 5]{FG} or \cite[(13)]{Petr}), for any $0<\theta<1$,
\begin{align*}
    \P(S_n > \theta \E[S_n]) \geq (1-\theta)^2\frac{\E[S_n]^2}{\E[S_n^2]} \geq \frac{(1-\theta)^2}{2}.
\end{align*}
\end{proof}

By combining Lemmas \ref{f25.1} and \ref{f25.10} we see that, with positive probability,  $S_n$ is at least of the order of $\E[S_n]$, which is at least of the order of $n^{3/4}$. We will improve this result and show that the probability can be arbitrarily close to 1 using  \eqref{f25.7}.

\begin{lemma}\label{f26.1}
    For every $\eps>0$ there exist $\delta>0$ and $n_0$ such that for  $n\geq n_0$,
\begin{align*}
    \P(S_n\leq \delta \E[S_n]) \leq \eps.
\end{align*}
\end{lemma}
\begin{proof}
We start with a heuristic outline of the proof.
We will subdivide the interval $[0,n]$ into $m$ equally long subintervals with endpoints $t_i = i(n/m)$, $0\leq i\leq m$,
where $m\geq 1$ will be determined later.
In the main part of the proof, we will assume that $n$ is divisible by $m$ so that $t_i$'s are integers.
We will use \eqref{f25.7} to show that 
 $S_{t_i}$ is larger than $\delta \E[S_n]$, where $\delta$ dependents on $m$,
 with probability bounded below no matter whether the events $\{S_{t_j}\geq \delta \E[S_n]\}$ occurred or not
 for $j<i$. This implies that the probability $\{S_{t_m}\geq \delta \E[S_n]\}$ goes to 1 exponentially fast as $m$ increases.\\

Fix an arbitrary $\eps>0$.
To start the rigorous proof, we apply expectation to both sides of \eqref{f25.6} and take $i=1$ and $j=k+1$
to obtain  $\E[H_k]\geq \E[H_{k+1}]$ for all $k\geq 0$. This and the linearity of expectation imply that 
$\E[S_{n/m}] \geq \E[S_n]/m$ for any positive integer $m$. Taking $\delta = 1/(2m)$, we get
\begin{align*}
    \P(S_n\leq \delta \E[S_n]) \leq \P(S_n\leq \E[S_{n/m}]/2).
\end{align*}
Hence, it will suffice to find an $m$ such that $\P(S_n\leq \E[S_{n/m}]/2) \leq \eps$. 

Recall that $t_i = i(n/m)$ for $0\leq i\leq m$, and let
\begin{align*}
    A_i = \{S_{t_i} \leq \E[S_{n/m}]/2\} \in \mathcal{F}_{t_i}
\end{align*}
for $0\leq i \leq m$. Note that $A_m=\{S_n\leq \E[S_{n/m}]/2\}$ so we only need to prove that $\P(A_m) \leq \eps$. 

We have $S_0 =1$.
It is easy to see that $\P(S_n =2)>0$ for all $n\geq 1$. 
Lemma \ref{f25.1} implies that for any fixed $m$ there exists $n_1$ such that for $n\geq n_1$,
$\E[S_{n/m}]/2 \geq 2$. This implies that $\P(A_i)>0$ for all $0\leq i \leq m$ if $n\geq n_1$.

Note that $A_{i}\subseteq A_{i-1}$ for all $1\leq i \leq m$ since $S_{t_i}$ is nondecreasing. Therefore,
\begin{align}\label{f26.5}
    \P(A_m) = \prod_{i=1}^{m} \P(A_i\mid A_{i-1}).
\end{align}
For any $1\leq i \leq m$, since $A_{i-1} \in \mathcal{F}_{t_{i-1}}$, \eqref{f25.7} implies that
\begin{align}\label{f26.6}
    \P(A_i \mid A_{i-1}) &=
    \P(S_{t_i} \leq \E[S_{n/m}]/2  \mid A_{i-1}) \leq \P(S_{t_1} \leq \E[S_{n/m}]/2  ) 
    = \P(A_1).
\end{align}
By Lemma \ref{f26.1} with $\theta = 1/2$, we have $\P(A_1)\leq 1-1/8 = 7/8$. This, \eqref{f26.5} and \eqref{f26.6} imply that
\begin{align*}
    \P(A_m) \leq \prod_{i=1}^{m} \P(A_1) \leq \left(7/8\right)^{m}.
\end{align*}
We now take $m> \log (\eps)/\log(7/8) $ and $\delta<1/(2m)$ to obtain  $\P(A_m) \leq \eps$.
This yields $\P(S_n\leq \delta \E[S_n]) \leq \eps$ for $n\geq n_1(m)$ that are divisible by $m$.

We pointed out earlier in the proof that  $\E[H_k]\geq \E[H_{k+1}]$ for all $k\geq 0$. This implies that 
$(n/(n-k))\E[S_{n-k}] \geq \E[S_n]$ for  $0\leq k\leq n$. Consider any $n\geq n_1(m)+m$ and let $0\leq k\leq m-1$ be such that
$n-k$ is divisible by $m$. Let $\delta' = \delta n_1(m)/(n_1(m)+m)$ so that $\delta' (n/(n-k)) \leq \delta $ for all $n\geq n_1(m)+m$ and  $0\leq k\leq m-1$. Then
\begin{align*}
    \P(S_n\leq \delta' \E[S_n]) &\leq \P(S_{n-k}\leq \delta' \E[S_n])
    \leq \P(S_{n-k}\leq \delta' (n/(n-k))\E[S_{n-k}])\\
    &\leq \P(S_{n-k}\leq \delta \E[S_{n-k}])
    \leq \eps.
\end{align*}
This proves the lemma with $\delta'$ in place of $\delta$.
\end{proof}

\begin{proof}[Proof of Theorem \ref{f24.1}]
Fix any $\eps>0$.
By Lemma \ref{f26.1} we can find
 $\delta_1>0$ such that $\P(S_n\geq \delta_1 \E[S_n]) \geq 1- \eps$ for sufficiently large $n$. 
 Lemma \ref{f25.1} implies that there exists  $c_2>0$ such that  $\E[S_n] \geq c_2 n^{3/4}$ for sufficiently large $n$. Hence
 for $\delta=c_2 \delta_1$,
\begin{align*}
    \liminf_{n\to \infty}\P(S_n \geq \delta n^{3/4}) \geq \liminf_{n\to\infty} \P(S_n\geq (\delta/c_2) \E[S_n])
    =\liminf_{n\to\infty} \P(S_n\geq \delta_1 \E[S_n]) \geq 1- \eps.
\end{align*}
\end{proof}

\begin{remark}
Our proof applies verbatim to the analogous model in dimensions higher than 2 but the model and the main result seem to be less interesting in higher dimensions. The crucial estimate is \eqref{m4.1}. In dimension 3,
by \cite[(2.35)]{Lawler} there exists a constant $c_*>0$ such that
\begin{align*}
    \liminf_{n\rightarrow \infty} (\log n)^{1/2}\E[\bone(C_n^{\text{right}})] > c_*.
\end{align*}
Simple random walk is transient in dimensions higher than 2, so for dimensions 4 and higher,
\begin{align*}
    \liminf_{n\rightarrow \infty} \E[\bone(C_n^{\text{right}})] > c_*.
\end{align*}
Therefore, Theorem \ref{f24.1} can be extended to higher dimensions $d$ as follows.
For every $\eps>0$ there exists $\delta>0$ such that
\begin{align*}
   & \liminf_{n\to\infty}\P(S_n \geq \delta n(\log n)^{-1/2}) \geq 1-\eps,
\qquad  d=3,\\
& \liminf_{n\to\infty}\P(S_n \geq \delta n) \geq 1-\eps,
\qquad  d\geq 4 .
\end{align*}

\end{remark}

\section{Open problems and conjectures}\label{open}

The earthworm model 
is very simple but it seems to be rather hard to analyze. We will present some
simulation results and open problems and conjectures based on the simulations.

\subsection{Dimension of the set of holes}

Our simulations of the set of holes created by the earthworm suggest that $S_n \sim n^\alpha$ with $\alpha \approx 0.79$.
This is consistent with the simulation results in \cite{burdzy2013fractal,wxml}. 
Fig. \ref{fig1} shows the results of simulations of $S_n$ for $n = 10^4, 3\cdot 10^4, 10^5, 3\cdot 10^5, 10^6, 3\cdot 10^6, 10^7$. For each $n$, we generated 10 i.i.d. samples of $S_n$ and calculated their means $\ol \mu_n$ (estimates of $\E S_n$). 
The regression line for $(\log(\ol \mu_n), \log (n))$ is $y=0.79x+0.06$.
Fig. \ref{fig1} shows the plot of the values of $(\log(\ol \mu_n), \log (n))$ and the regression line.
\begin{conjecture}\label{con34}
$\liminf_{n\to\infty} \log(\E[S_n]) /\log(n) > 3/4$.
\end{conjecture}
The above conjecture is a mild version of the following apparently very hard problem.
\begin{problem}
    Prove existence and determine the value of $\lim_{n\to\infty} \log(\E[S_n]) /\log(n) $.
\end{problem}

\begin{figure}
  \centering
  \includegraphics[scale=0.6]{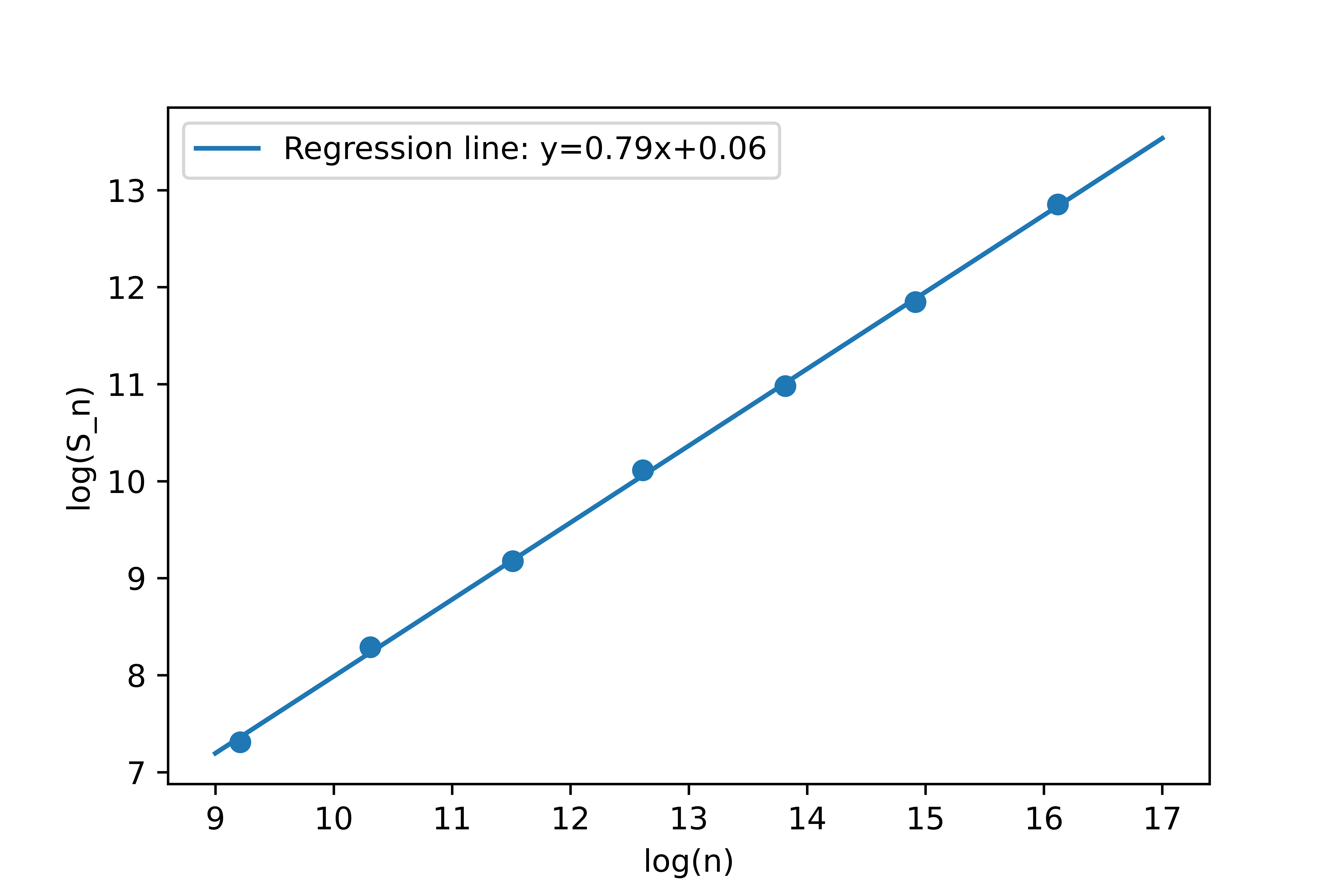}
  \caption{Plot of  $(\log(\ol \mu_n), \log (n))$ and the regression line $y=0.79x+0.06$.}
  \label{fig1}
\end{figure}

\subsection{Location of holes}

Fig. \ref{fig2} shows the simulation of hole locations, i.e., the set $\mathcal{H}_n$, after  $n = 10^8$ steps.

\begin{problem}
(i)
Find the distribution of sizes of connected components of $\mathcal{H}_n$.

(ii)
Find the distribution of sizes of connected components of the ``complement'' of $\mathcal{H}_n$
in the trail of the earthworm,
i.e., $\bigcup_{1\leq k\leq n} \{X_k\}\setminus \mathcal{H}_n$.
\end{problem}

\begin{figure}
  \centering
  \includegraphics[scale=0.11]{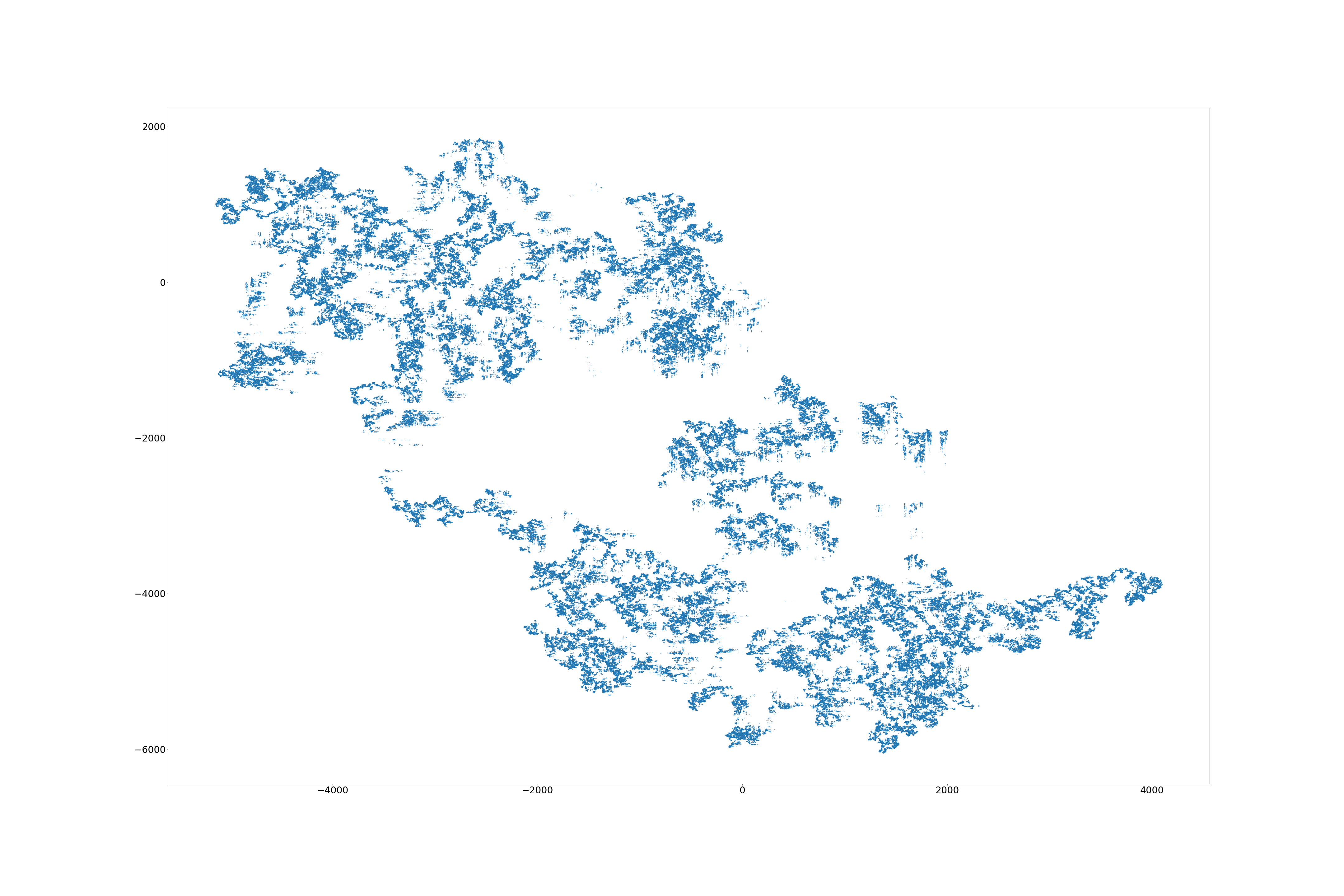}
  \caption{Locations of holes $\mathcal{H}_n$ after $n = 10^8$ steps.}
  \label{fig2}
\end{figure}

\subsection{Central Limit Theorem}

The following remarks on the CLT for $S_n$ are highly speculative in view of the fact that we even do not have  a good understanding of the mean of $S_n$. Nevertheless, we present our simulation results in Fig. \ref{fig3}.

We set $n=10,000$ and generated 2,000 i.i.d. samples of $S_n$. 
The histogram in Fig. \ref{fig3} suggests that CLT may hold for $S_n$.
The Kolmogorov–Smirnov test (see \cite{massey1951kolmogorov}) yields the statistic equal to $0.0214$ which gives the $p$-value equal to $0.315$, supporting the CLT conjecture.

\begin{figure}
  \centering
 \includegraphics[scale=0.5]{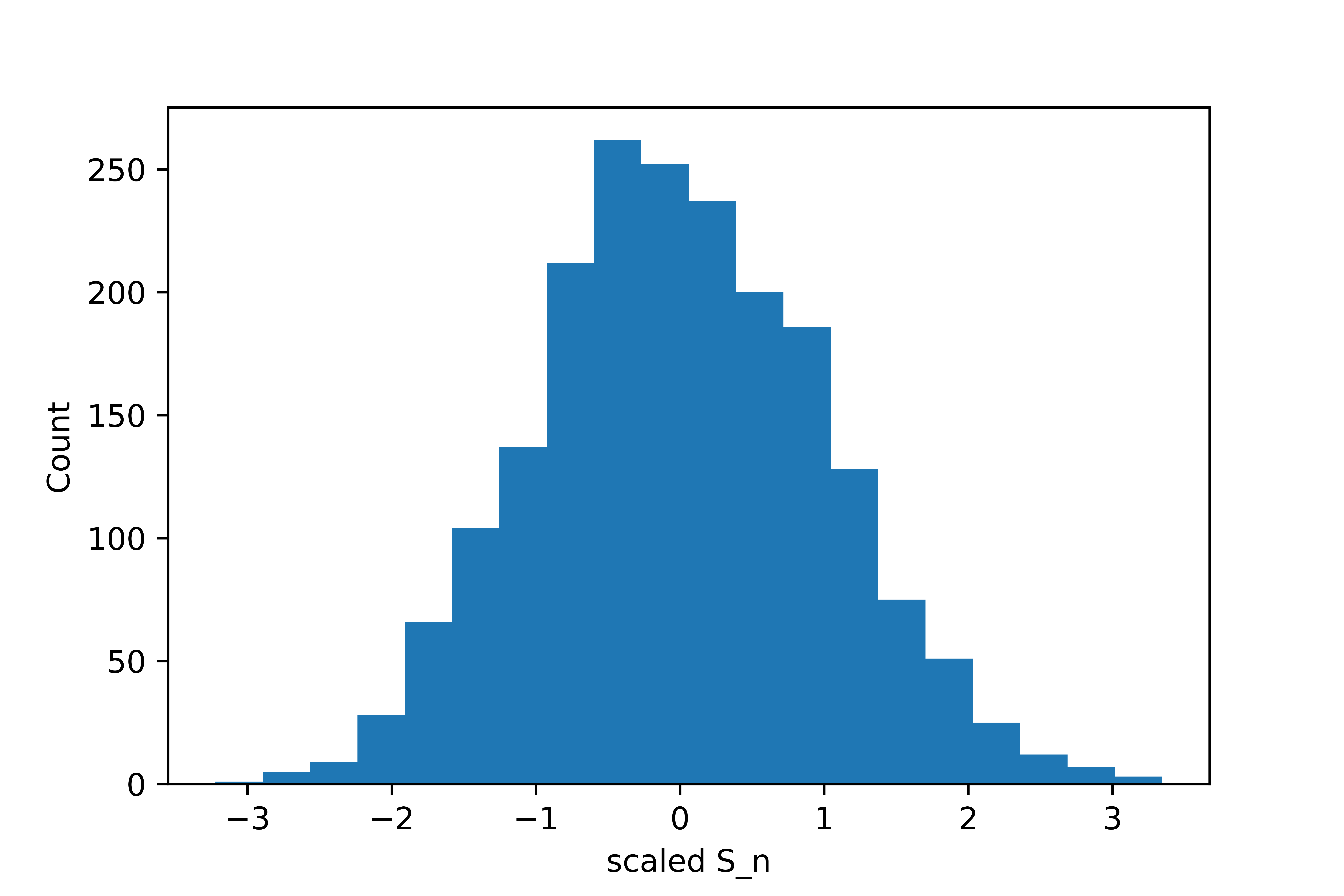}
  \caption{Normalized histogram of 2,000 i.i.d. samples of $S_n$ with $n=10,000$.}
  \label{fig3}
\end{figure}

\bibliographystyle{alpha}

\newcommand{\etalchar}[1]{$^{#1}$}

\end{document}